\documentclass{amsart}

\newtheorem{theorem}{Theorem}[section]
\newtheorem{lemma}[theorem]{Lemma}

\theoremstyle{definition}

\theoremstyle{remark}
\newtheorem{remark}[theorem]{Remark}

\numberwithin{equation}{section}

\begin{document}

\title{Harnack inequality\\for the negative power Gaussian curvature flow}

\author{Yi Li}

\address{Department of Mathematics, Harvard University, Cambridge, MA, 02138}
\email{yili@math.harvard.edu}
%%%%%%%%%%%%%%%%%%%%%%%%%%%%%%%%%%%%%%%%%%%%%%%%%%%%%%%%%%%%%%%%%%%%%%%%%%%%%%


\date{Received by the editors August 29, 2010}

\dedicatory{{\rm (Communicated by Jianguo Cao)}}


\begin{abstract} In this paper, we study the power of Gaussian curvature flow of a compact convex hypersurface and establish its Harnack inequality when the power
is negative. In the Harnack inequality, we require that the absolute value of the power is strictly positive and strictly less
than the inverse of the dimension of the hypersurface.
\end{abstract}

\maketitle

%%%%%%%%%%%%%%%%%%%%%%%%%%%%%%%%%%%%%%%%%%%%%%%%%%%%%%%%%%%%%%%%%%%%%%%%%%%%%%
\section{Introduction}
%%%%%%%%%%%%%%%%%%%%%%%%%%%%%%%%%%%%%%%%%%%%%%%%%%%%%%%%%%%%%%%%%%%%%%%%%%%%%%
The Harnack estimate or Harnack inequality plays an important role in geometric 
flows. For the heat
equation, P. Li and S.-T. Yau \cite{LY} obtained the corresponding Harnack inequality
by using the parabolic maximum principle. Hamilton \cite{H1, H2} proved a Harnack inequality for the Ricci flow and mean curvature flow for all dimensions. For a Harnack
inequality for $m$-power mean
curvature flow, we refer to \cite{A}, \cite{S} and \cite{W}, where $m$ is positive.
B. Chow \cite{C2} considered a Harnack inequality for $m$-power Gaussian
curvature flow for $m>0$.

In this paper, we consider the negative power Gaussian curvature flow of a compact convex hypersurface $F_{0}: M^{n}\to\mathbb{R}^{n+1}$,
\begin{equation}
\frac{\partial}{\partial t}F(x,t)=\frac{1}{K(x,t)^{b}}\cdot\nu(x,t), \
0<b<\frac{1}{n}; \
F(x,0)=F_{0}(x), \ x\in M^{n}.\label{1.1}
\end{equation}
Here $K$ is the Gaussian curvature and $\nu$ denotes the outward unit normal vector field. Using the similar argument in \cite{C2}, we obtain a Harnack inequality for
the flow (\ref{1.1}).

\begin{theorem} \label{t1.1}Suppose that $F_{0}: M^{n}\to\mathbb{R}^{n+1}$ is a compact convex hypersurface. If $0<b<\frac{1}{n}$, then
\begin{equation}
\frac{\partial}{\partial t}\left(\frac{1}{K(x,t)^{b}}\right)+\left|\nabla\left(\frac{1}{K(x,t)^{b}}\right)\right|^{2}_{h}
-\frac{nb}{(1-nb)t}\left(\frac{1}{K(x,t)^{b}}\right)\leq0,\label{1.2}
\end{equation}
where the notation $|\cdot|_{h}$ is defined in the next section.
\end{theorem}

When $\frac{1}{n}\leq b\leq1$, some interesting results have been derived 
in \cite{L}, where the author considered $n=2$.

%%%%%%%%%%%%%%%%%%%%%%%%%%%%%%%%%%%%%%%%%%%%%%%%%%%%%%%%%%%%%%%%%%%%%%%%%%%%%%
\section{Notation and evolution equations}
%%%%%%%%%%%%%%%%%%%%%%%%%%%%%%%%%%%%%%%%%%%%%%%%%%%%%%%%%%%%%%%%%%%%%%%%%%%%%%

%%%%%%%%%%%%%%%%%%%%%%%%%%%%%%%%%%%%%%%%%%%%%%%%%%%%%%%%%%%%%%%%%%%%%%%%%%%%%%
\subsection{Notation}
%%%%%%%%%%%%%%%%%%%%%%%%%%%%%%%%%%%%%%%%%%%%%%%%%%%%%%%%%%%%%%%%%%%%%%%%%%%%%%
Suppose that $F: M^{n}\to\mathbb{R}^{n+1}$ is a hypersurface. The second fundamental form is given by
\begin{equation}
h_{ij}=-\left\langle\frac{\partial^{2}F}{\partial x^{i}\partial x^{j}},\nu\right\rangle,\label{2.1}
\end{equation}
where $\langle\cdot,\cdot\rangle$ denotes the standard metric on $\mathbb{R}^{n+1}$.
The induced metric
\begin{equation}
g_{ij}=\left\langle\frac{\partial F}{\partial x^{i}},
\frac{\partial F}{\partial x^{j}}\right\rangle\label{2.2}
\end{equation}
on $M$ gives us the mean curvature
\begin{equation}
H=g^{ij}h_{ij}\label{2.3}
\end{equation}
and the Gaussian curvature
\begin{equation}
K=\frac{{\rm det}(h_{ij})}{{\rm det}(g_{ij})}.\label{2.4}
\end{equation}

If $\alpha=\{\alpha_{i}\}$ and $\beta=\{\beta_{i}\}$ are $1$-forms and $s=\{s_{ij}\}$ is a symmetric positive
definite covariant $2$-tensor, we use the short notation
\begin{equation*}
\langle \alpha,\beta\rangle_{s}\equiv\langle \alpha_{i},\beta_{i}\rangle_{s}:=s^{-1}_{ij}\alpha_{i}\beta_{j},
\end{equation*}
where $(s^{-1}_{ij})$ is the inverse matrix of $(s_{ij})$. Similarly, if $A=\{\alpha_{ijk}\}$ and $B=\{\beta_{pqr}\}$ are covariant $3$-tensors, we define
\begin{equation*}
\langle A,B\rangle_{s}\equiv\langle A_{ijk},B_{ijk}\rangle_{s}:=s^{-1}_{ip}s^{-1}_{jq}s^{-1}_{kr}A_{ijk}B_{pqr}.\label{3.2.6}
\end{equation*}
Finally, we define the Laplacian-type operator by
\begin{equation}
\Box:=h^{-1}_{ij}\nabla_{i}\nabla_{j}.\label{2.5}
\end{equation}
Here $\nabla$ denotes the Levi-Civita connection of the induced metric $g$ on $M$.

Let $M^{n}$ be a convex hypersurface in $\mathbb{R}^{n+1}$, $\alpha=\{\alpha_{i}\}$
a $1$-form on $M^{n}$, and $\phi$ a smooth function on $M$. We have the following identities (see \cite{C1}
or \cite{C2}):
\begin{eqnarray}
R_{ijk\ell}&=&h_{i\ell}h_{jk}-h_{ik}h_{j\ell},\label{2.6}\\
\nabla_{i}h_{jk}&=&\nabla_{j}h_{ik},\label{2.7}\\
(\nabla_{i}\nabla_{j}-\nabla_{j}\nabla_{i})\alpha_{k}
&=&-R_{ijk\ell}g^{\ell p}\alpha_{p},\label{2.8}\\
\nabla_{i}\nabla_{j}F&=&-h_{ij}\nu,\label{2.9}\\
\nabla_{i}\nu&=&h_{ij}g^{jk}\nabla_{k}F,\label{2.10}\\
\nabla_{i}\nabla_{j}\nu&=&g^{k\ell}\nabla_{k}h_{ij}\cdot\nabla_{\ell}F
-g^{\ell k}h_{i\ell}h_{kj}\nu,\label{2.11}\\
\nabla_{i}\left(Kh^{-1}_{ij}\right)&=&0,\label{2.12}\\
\left(\nabla_{i}\Box-\Box\nabla_{i}\right)\phi&=&-\langle\nabla_{i}h_{jk},
\nabla_{j}\nabla_{k}\phi\rangle_{h}
-(n-1)h_{ij}g^{jk}\nabla_{k}\phi.\label{2.13}
\end{eqnarray}

%%%%%%%%%%%%%%%%%%%%%%%%%%%%%%%%%%%%%%%%%%%%%%%%%%%%%%%%%%%%%%%%%%%%%%%%%%%%%%
\subsection{Evolution equations}
%%%%%%%%%%%%%%%%%%%%%%%%%%%%%%%%%%%%%%%%%%%%%%%%%%%%%%%%%%%%%%%%%%%%%%%%%%%%%%
Now we consider a generalized Gaussian curvature flow
\begin{equation}
\frac{\partial}{\partial t}F(x,t)=-f(K(x,t))\cdot\nu(x,t), \ \ \ F(x,0)=F_{0}(x), \
x\in M^{n},\label{2.14}
\end{equation}
where $f: (0,+\infty)\to\mathbb{R}$ is a smooth function depending only on the Gaussian
curvature $K$, which satisfies $f'>0$ everywhere in order to guarantee 
a short time existence. Such a type of Gaussian curvature flow is called the {\it
$f$-Gaussian curvature flow}.

\begin{remark} \label{r2.1}For convenience, in what follows, we write
$f_{t}=f(K_{t})$ and $\partial_{t}=\frac{\partial}{\partial t}$.
\end{remark}

Under the $f$-Gaussian curvature flow, it is easy to verify the following evolution equations (compared with Lemma 3.1 in \cite{C2}), where $h_{t}=\{(h_{t})_{ij}\}$:
\begin{eqnarray}
\partial_{t}(g_{t})_{ij}&=&-2f_{t}(h_{t})_{ij},\label{2.15}\\
\partial_{t}\nu_{t}&=&\nabla f_{t} \ \ = \ \ f'_{t}
\cdot\nabla K_{t},\label{2.16}\\
\partial_{t}(h_{t})_{ij}&=&\nabla_{i}\nabla_{j}f_{t}-f_{t}(g_{t})^{k\ell}
(h_{t})_{ik}(h_{t})_{\ell j},\label{2.17}\\
\partial_{t}K_{t}&=&f'_{t}K_{t}\cdot\left(\Box_{t} K_{t}+
\frac{f''_{t}}{f'_{t}}|\nabla K_{t}|^{2}_{h_{t}}+\frac{f_{t}}{f'_{t}K_{t}}H_{t}K_{t}\right),\label{2.18}\\
\partial_{t}f_{t}&=&f'_{t}K_{t}\cdot\left[\Box_{t} f_{t}+H_{t}f_{t}\right],\label{2.19}\\
\partial_{t}H_{t}&=&\Delta_{t} f_{t}+f_{t}|h_{t}|^{2}_{g_{t}},\label{2.20}\\
\partial_{t}\Box_{t}&=&-\langle\nabla\nabla f_{t},
\nabla\nabla\rangle_{h_{t}}
+f_{t}\Delta_{t}+\left(2-n+\frac{f_{t}}{f'_{t}K_{t}}\right)\langle\Delta_{t} f_{t},
\Delta_{t}\rangle_{g}.\label{2.21}
\end{eqnarray}

\begin{remark} \label{r2.2}If $M^{n}$ is compact and convex, then $H=H_{0}>0$;
using the evolution equation (\ref{2.20}), we see that $H(x,t)=H_{t}(x)>0$ under the $f$-Gaussian curvature
flow. According to (\ref{2.4}), we conclude that $K(x,t)>0$ along the
$f$-Gaussian curvature flow. Therefore $\frac{1}{K_{t}}$ is well-defined.
\end{remark}

%%%%%%%%%%%%%%%%%%%%%%%%%%%%%%%%%%%%%%%%%%%%%%%%%%%%%%%%%%%%%%%%%%%%%%%%%%%%%%
\section{Harnack inequality}
%%%%%%%%%%%%%%%%%%%%%%%%%%%%%%%%%%%%%%%%%%%%%%%%%%%%%%%%%%%%%%%%%%%%%%%%%%%%%%
Motivated by the self-similar solutions in \cite{C2}, we define a time-dependent
tensor field $P_{t}=\{(P_{t})_{ij}\}$ by
\begin{equation}
(P_{t})_{ij}=\nabla_{i}\nabla_{j}f_{t}-(h_{t})^{-1}_{k\ell}
\nabla_{k}(h_{t})_{ij}\cdot\nabla_{\ell}f_{t}+f_{t}(g_{t})^{k\ell}(h_{t})_{ik}(h_{t})_{\ell j}.\label{3.1}
\end{equation}
Taking the trace of $(P_{t})_{ij}$ with respect to $(h_{t})_{ij}$, we set
\begin{equation}
\mathbf{P}_{t}=(h_{t})^{-1}_{ij}(P_{t})_{ij}.\label{3.2}
\end{equation}
Since $\nabla K_{t}=K_{t}(h_{t})^{-1}_{pq}\nabla (h_{t})_{pq}$ by (\ref{2.12}), we can rewrite
$\mathbf{P}_{t}$ as
\begin{equation}
\mathbf{P}_{t}=\Box_{t}f_{t}+f_{t}H_{t}-(h_{t})^{-1}_{ij}(h_{t})^{-1}_{k\ell}
\nabla_{k}(h_{t})_{ij}\cdot\nabla_{\ell}f_{t}=\Box_{t} f_{t}+f_{t}H_{t}-\frac{|\nabla f_{t}|^{2}_{h_{t}}}{f'_{t}K_{t}}.\label{3.3}
\end{equation}
%%%%%%%%%%%%%%%%%%%%%%%%%%%%%%%%%%%%%%%%%%%%%%%%%%%%%%%%%%%%%%%%%%%%%%%%%%%%%%
\subsection{Evolution equation for $\mathbf{P}_{t}$}
%%%%%%%%%%%%%%%%%%%%%%%%%%%%%%%%%%%%%%%%%%%%%%%%%%%%%%%%%%%%%%%%%%%%%%%%%%%%%%
In this subsection our task is to find the evolution equation for $\mathbf{P}_{t}$. 
Before doing this, we first write down some elementary formulas which will be used in
our complicated and tedious computation. Since $f$ is
smooth depending only on $K_{t}$, we have $\nabla f_{t}=f'_{t}\nabla K_{t}$ and
\begin{eqnarray}
\Box_{t} f_{t}&=&(h_{t})^{-1}_{ij}\nabla_{i}\nabla_{j}f_{t} \ \ = \ \ (h_{t})^{-1}_{ij}\nabla_{i}
\left(f'_{t}\nabla_{j}K_{t}\right)\nonumber\\
&=&(h_{t})^{-1}_{ij}\left[f''_{t}\nabla_{i}K_{t}\cdot\nabla_{j}K_{t}
+f'_{t}\nabla_{i}\nabla_{j}K_{t}\right]\label{3.4} \ \ = \ \ f'_{t}\Box_{t} K_{t}+f''_{t}|\nabla K_{t}|^{2}_{h_{t}}.
\end{eqnarray}
Hence, (\ref{2.18}) can be rewritten as
\begin{equation*}
\partial_{t}K_{t}=K_{t}\left(\Box_{t}f_{t}+f_{t}H_{t}\right).
\end{equation*}
Using $\nabla_{i}K_{t}=\nabla_{i}f_{t}/f'_{t}$, we obtain
\begin{equation}
\nabla_{i}\nabla_{j}K_{t}=\nabla_{i}\left(f'_{t}{}^{-1}\nabla_{i}f_{t}
\right)=-f'_{t}{}^{-3}f''_{t}\nabla_{i}f_{t}\cdot\nabla_{j}f_{t}
+f'_{t}{}^{-1}\nabla_{i}\nabla_{j}f_{t}.\label{3.5}
\end{equation}
The next useful formula is
\begin{eqnarray}
\Box_{t}(f'_{t})K_{t})&=&(h_{t})^{-1}_{ij}\nabla_{i}\left(f''_{t}\nabla_{j}K_{t}\cdot K_{t}+f'_{t}\cdot\nabla_{j}K_{t}\right)\nonumber\\
&=&(h_{t})^{-1}_{ij}\left[f''_{t}K_{t}\nabla_{i}\nabla_{j}K_{t}\right.+\left.(f'''_{t}K_{t}+2f''_{t})\nabla_{i}K_{t}\cdot\nabla_{j}K_{t}
+f'_{t}\nabla_{i}\nabla_{j}K_{t}\right]\nonumber\\
&=&(h_{t})^{-1}_{ij}\left[f'_{t}+f''_{t}K_{t}\right]
\left(\frac{1}{f'_{t}}\nabla_{i}\nabla_{j}f_{t}-\frac{f''_{t}}{f'_{t}{}^{3}}
\nabla_{i}f_{t}\cdot\nabla_{j}f_{t}\right)\nonumber\\
&&+(h_{t})^{-1}_{ij}\left[f'''_{t}K_{t}+2f''_{t}\right]
\frac{\nabla_{i}f_{t}\cdot\nabla_{j}f_{t}}{f'_{t}{}^{2}}\nonumber\\
&=&\left(1+\frac{f''_{t}K_{t}}{f'_{t}}\right)\Box_{t}f_{t}
+\left(\frac{f'''_{t}K_{t}}{f'_{t}{}^{2}}+\frac{f''_{t}}{f'_{t}{}^{2}}
-\frac{f''_{t}{}^{2}K_{t}}{f'_{t}{}^{3}}\right)|\nabla f_{t}|^{2}_{h_{t}}.
\label{3.6}
\end{eqnarray}

\begin{lemma} \label{l3.1}Under the $f$-Gaussian curvature flow, we have
\begin{eqnarray}
\partial_{t}\left(\Box_{t} f_{t}\right)&=&f'_{t}K_{t}\cdot\Box_{t}(\Box_{t}
 f_{t})+2\left(1+
\frac{f''_{t}K_{t}}{f'_{t}}\right)\langle\nabla f_{t},\nabla(\Box_{t} f_{t}\rangle_{h_{t}}\nonumber\\
&&+\left(1+\frac{f''_{t}K_{t}}{f'_{t}}\right)(\Box_{t} f_{t})^{2}+f'_{t}K_{t}\cdot\Box_{t}(H_{t}f(K_{t}))\nonumber\\
&&+2\left(1+\frac{f''_{t}K_{t}}{f'_{t}}\right)\langle\nabla f_{t},\nabla(H_{t}f_{t})\rangle_{h_{t}}+\left(1
+\frac{f''_{t}K_{t}}{f'_{t}}\right)H_{t}f_{t}\cdot\Box_{t} f_{t}\label{3.7}\\
&&+\left(\frac{f'''_{t}K_{t}}{f'_{t}{}^{2}}
+\frac{f''_{t}}{f'_{t}{}^{2}}-\frac{f''_{t}{}^{2}K_{t}}{f'_{t}{}^{3}}
\right)
|\nabla f_{t}|^{2}_{h_{t}}\left[\Box_{t} f_{t}+H_{t}f_{t}\right]\nonumber\\
&&-|\nabla\nabla f_{t}|^{2}_{h_{t}}+f_{t}\Delta_{t} f_{t}+\left(2-n+\frac{f_{t}}{f'_{t}K_{t}}\right)|\nabla f_{t}|^{2}_{g_{t}}.\nonumber
\end{eqnarray}
\end{lemma}

\begin{proof} From $\partial_{t}(\Box_{t} f_{t})=\left(\partial_{t}\Box_{t}
\right)f_{t}+\Box_{t}\left(\partial_{t}f_{t}\right)$, we get
\begin{equation*}
\partial_{t}(\Box_{t} f_{t})=-|\nabla\nabla f_{t}|^{2}_{h_{t}}
+f_{t}\Delta_{t} f_{t}+\left(2-n+\frac{f_{t}}{f'_{t}K_{t}}\right)|\nabla f_{t}|^{2}_{g_{t}}+\Box_{t}\left[f'_{t}K_{t}(\Box_{t} f_{t}+H_{t}
f_{t})\right].
\end{equation*}
Now we evaluate the last term:
\begin{eqnarray*}
\Box_{t}\left[f'_{t}K_{t}(\Box_{t} f_{t}+H_{t}f_{t})\right]&=&\Box_{t}(f'_{t}K_{t})\cdot(\Box_{t} f_{t}+H_{t}f_{t})+f'_{t}K_{t}\cdot\Box_{t}(\Box_{t} f_{t}+H_{t}f_{t})\\
&&+2\langle\nabla(f'_{t}K_{t}),\nabla(\Box_{t} f_{t}+H_{t}f_{t})\rangle_{h_{t}}\\
&=&\left(1+\frac{f''_{t}K}{f'_{t}}\right)\Box_{t} f_{t}\left(\Box_{t}f_{t}+H_{t}f_{t}\right)\\
&&+\left(\frac{f'''_{t}K_{t}}{f'_{t}{}^{2}}
+\frac{f''_{t}}{f'_{t}{}^{2}}-\frac{f''_{t}{}^{2}K_{y}}{f'_{t}{}^{3}}
\right)|\nabla f_{t}|^{2}_{h_{t}}(\Box_{t} f_{t}+H_{t}f_{t})\\
&&+f'_{t}K_{t}\cdot\Box_{t}(\Box_{t} f_{t})+f'_{t}K_{t}\cdot\Box_{t}(H_{t}f_{t})\\
&&+2\left\langle\left(1+\frac{f''_{t}K_{t}}{f'_{t}}\right)\nabla f_{t},
\nabla(\Box_{t} f_{t})+\nabla(H_{t}f_{t})\right\rangle_{h_{t}}.
\end{eqnarray*}
Simplifying the above and plugging into the expression of
$\partial_{t}(\Box_{t} f_{t})$, we obtain the required result.
\end{proof}

\begin{lemma} \label{l3.2}Under the $f$-Gaussian curvature flow, we have
\begin{eqnarray*}
\partial_{t}\left(\frac{|\nabla f_{t}|^{2}_{h}}{f'_{t}K_{t}}\right)&=&
(f'_{t}K_{t})\Box_{t}\left(\frac{|\nabla f_{t}|^{2}_{h_{t}}}{f'_{t}K_{t}}\right)+2\left(1+\frac{f''_{t}K_{t}}{f'_{t}}\right)\left\langle\nabla f_{t},\nabla
\left(\frac{|\nabla f_{t}|^{2}_{h_{t}}}{f'_{t}K_{t}}\right)\right\rangle_{h_{t}}\\
&&-2|\nabla\nabla f_{t}|^{2}_{h_{t}}+2\langle\nabla_{i}(h_{t})_{jk},
\nabla_{i}f_{t}\cdot
\nabla_{j}\nabla_{k}f_{t}\rangle_{h_{t}}\\
&&+2f_{t}\langle\nabla H_{t},\nabla f_{t}\rangle_{h_{t}}+\left[1+\left(1+\frac{f''_{t}K_{t}}{f'_{t}}
\right)\frac{f_{t}}{f'_{t}K_{t}}\right]H_{t}|
\nabla f_{t}|^{2}_{h_{t}}\\
&&-\langle\nabla_{i}f\cdot\nabla_{j}h_{kl},\nabla_{j}f\cdot\nabla_{i}h_{kl}\rangle_{h}
+\left(1+\frac{f''_{t}K_{t}}{f'_{t}}\right)\frac{2\Box_{t} f_{t}|\nabla f_{t}|^{2}_{g}}{f'_{t}K_{t}}\\
&&+\left(\frac{f}{f'K}+2-n\right)|\nabla f|^{2}_{g}-\left[\frac{f''_{t}{}^{2}}{f'_{t}{}^{4}}
-\frac{f'''_{t}}{f'_{t}{}^{3}}+\frac{1}{(f'_{t}K_{t})^{2}}\right]|\nabla f_{t}
|^{4}_{h_{t}}.
\end{eqnarray*}
\end{lemma}

\begin{proof} The proof is similar to that in \cite{C2}. We observe first that
\begin{eqnarray*}
\partial_{t}\left(\frac{|\nabla f_{t}|^{2}_{h_{t}}}{f'_{t}K_{t}}\right)&=&\partial_{t}
\left(f'_{t}{}^{-1}K^{-1}_{t}(h_{t})^{-1}_{ij}\nabla_{i}f_{t}\nabla_{j}f_{t}\right)\\
&=&-\frac{1}{f'_{t}{}^{2}}\left(\frac{f''_{t}}{f'_{t}K_{t}}+\frac{1}{K^{2}_{t}}
\right)f'_{t}K_{t}(\Box_{t} f_{t}+H_{t}f_{t})|\nabla f_{t}|^{2}_{h_{t}}
\\
&&+2(f'_{t}K_{t})^{-1}\langle\nabla(f'_{t}K_{t}(\Box_{t} f_{t}+H_{t}f_{t})),\nabla f_{t}\rangle_{h_{t}}\\
&&-(f'_{t}K_{t})^{-1}(h_{t})^{-1}_{ik}(h_{t})^{-1}_{j\ell}
\left(\nabla_{k}\nabla_{\ell}f_{t}-f_{t}g^{pq}h_{kp}h_{q\ell}\right)\nabla_{i}f_{t}
\nabla_{j}f_{t}\\
&=&-\frac{K_{t}}{f'_{t}}\left(\frac{f''_{t}}{f'_{t}K_{t}}+\frac{1}{K^{2}_{t}}
\right)
\left(\Box_{t} f_{t}+H_{t}f_{t}\right)|\nabla f_{t}|^{2}_{h_{t}}+\frac{f_{t}
|\nabla f_{t}|^{2}_{g_{t}}}{f'_{t}K_{t}}\\
&&-(f'_{t}K_{t})^{-1}\langle\nabla_{i}\nabla_{j}f_{t},\nabla_{i}f_{t}
\nabla_{j}f_{t}\rangle_{h_{t}}+2\langle\nabla(\Box_{t} f_{t}+H_{t}f_{t}),
\nabla f_{t}\rangle_{h_{t}}\\
&&+2(f'_{t}K_{t})^{-1}\langle\nabla(f'_{t}K_{t}),\nabla f_{t}
\rangle_{h_{t}}(\Box_{t} f_{t}+H_{t}f_{t})\\
&=&\left(1+\frac{f''_{t}K_{t}}{f'_{t}}\right)\frac{(\Box_{t} f_{t}
+H_{t}f_{t})|\nabla f_{t}|^{2}_{h_{t}}}{f'_{t}K_{t}}
-\frac{\langle\nabla_{i}\nabla_{j}f_{t},\nabla_{i}f_{t}
\nabla_{j}f_{t}\rangle_{h}}{f'_{t}K_{t}}\\
&&+\frac{f_{t}}{f'_{t}K_{t}}|\nabla f_{t}|^{2}_{g_{t}}
+2\langle\nabla(\Box_{t} f_{t}+H_{t}f_{t}),\nabla f_{t}\rangle_{h_{t}}.
\end{eqnarray*}
On the other hand, we compute the Laplacian of $\frac{|\nabla f_{t}
|^{2}_{h_{t}}}{f'_{t}K_{t}}$ with respect to $(h_{t})_{ij}$:
\begin{eqnarray*}
\Box_{t}\left((f'_{t}K_{t})^{-1}|\nabla f_{t}|^{2}_{h_{t}}\right)&=&\Box_{t}\left((f'_{t}K_{t})^{-1}\right)|\nabla f_{t}|^{2}_{h_{t}}
+(f'_{t}K_{t})^{-1}\cdot\Box_{t} (h_{t})^{-1}_{ij}
\cdot\nabla_{i}f_{t}\nabla_{j}f_{t}\\
&&+(h_{t})^{-1}_{ij}\frac{2\Box_{t}(\nabla_{i}f_{t})
\cdot\nabla_{j}f_{t}
+2\left\langle\nabla_{k}\nabla_{i}f_{t},\nabla_{k}\nabla_{j}f_{t}
\right\rangle_{h_{t}}}{f'_{t}K_{t}}\\
&&+4\left\langle\nabla_{k}h^{-1}_{ij},\nabla_{k}\nabla_{i}f\cdot\nabla_{j}f
\right\rangle_{h}(f'K)^{-1}\\
&&+2\left\langle\nabla_{k}(f'_{t}K_{t})^{-1},\nabla_{k}(h_{t})^{-1}_{ij}
\cdot\nabla_{i}f_{t}\nabla_{j}f_{t}\right\rangle_{h_{t}}\\
&&+4\left\langle\nabla_{k}(f'_{t}K_{t})^{-1},(h_{t})^{-1}_{ij}
\nabla_{k}\nabla_{i}f_{t}\cdot\nabla_{j}f_{t}\right\rangle_{h_{t}}\\
&=&\Box_{t}\left((f'_{t}K_{t})^{-1}\right)|\nabla f_{t}|^{2}_{h_{t}}+
(f'_{t}K_{t})^{-1}\cdot\Box (h_{t})^{-1}_{ij}\cdot\nabla_{i}f_{t}\nabla_{j}f_{t}
\\
&&+2(f'_{t}K_{t})^{-1}\langle\Box_{t}(\nabla f_{t}),\nabla f_{t}\rangle_{h_{t}}+2(f'_{t}K_{t})^{-1}|\nabla_{i}\nabla_{j}f_{t}|^{2}_{h_{t}}\\
&&+2\left\langle\nabla_{k}(f'_{t}K_{t})^{-1},\nabla_{k}(h_{t})^{-1}_{ij}
\right\rangle_{h_{t}}
\cdot\nabla_{i}f_{t}\nabla_{j}f_{t}\\
&&+4\left\langle\nabla_{k}(f'_{t}K_{t})^{-1}\nabla_{i}f_{t},\nabla_{k}
\nabla_{i}f_{t}\right\rangle_{h_{t}}\\
&&-4(f'_{t}K_{t})^{-1}\langle\nabla_{k}(h_{t})_{ij},\nabla_{k}\nabla_{i}f_{t}
\cdot\nabla_{j}f_{t}\rangle_{h_{t}}.
\end{eqnarray*}
We compute some elementary formulas which will be used later. Note that
\begin{eqnarray*}
\nabla(K^{-1}_{t})&=&-K^{-2}_{t}\nabla K_{t} \ \ = \ \
-\frac{\nabla f_{t}}{f'_{t}K^{2}_{t}},\\
\nabla(f'_{t}{}^{-1})&=&-f'_{t}{}^{-2}\nabla f'_{t} \ \ = \ \
-\frac{f''_{t}}{f'_{t}{}^{3}}\nabla f_{t}.
\end{eqnarray*}
Therefore
\begin{eqnarray*}
\Box_{t}(K^{-1}_{t})&=&(h_{t})^{-1}_{ij}\left(2K^{-3}_{t}\nabla_{i}K_{t}\cdot\nabla_{j}K_{t}
-K^{-2}_{t}\nabla_{i}\nabla_{j}K_{t}\right)\\
&=&(h_{t})^{-1}_{ij}\left[\frac{2}{K^{3}_{t}}\frac{\nabla_{i}f_{t}
\cdot\nabla_{j}f_{t}}{f'_{2}{}^{2}}
-\frac{1}{K^{2}_{t}}\left(\frac{\nabla_{i}\nabla_{j}f_{f}}{f'_{t}}
-\frac{f''_{t}}{f'_{t}{}^{3}}\nabla_{i}f_{t}\nabla_{j}f_{t}\right)\right]\\
&=&\frac{2}{f'_{t}{}^{2}K^{3}_{t}}|\nabla f_{t}|^{2}_{h_{t}}
-\frac{1}{f'_{t}K^{2}_{t}}\Box_{t} f_{t}+\frac{f''_{t}}{f'_{t}{}^{3}K^{2}_{t}}
|\nabla f_{t}|^{2}_{h_{t}}\\
&=&\left(\frac{2}{K_{t}}+\frac{f''_{t}}{f'_{t}}\right)\frac{1}{(f'_{t}K_{t})^{2}}|\nabla f_{t}|^{2}_{h_{t}}
-\frac{1}{f'_{t}K^{2}_{t}}\Box_{t} f_{t},\\
\Box_{t}(f'_{t}{}^{-1})&=&(h_{t})^{-1}_{ij}\nabla_{i}\left(-\frac{f''_{t}}{f'_{t}{}^{3}}\nabla_{j}
f_{t}\right)\\
&=&\left(\frac{3f''_{t}{}^{2}}{f'_{t}{}^{5}}-\frac{f'''_{t}}{f'_{t}{}^{4}}\right)|\nabla f_{t}|^{2}_{h_{t}}-\frac{f''_{t}}{f'_{t}{}^{3}}\Box_{t} f_{t},
\end{eqnarray*}
and
\begin{equation*}
\left\langle\nabla(f'_{t}{}^{-1}),\nabla(K^{-1}_{t})\right\rangle_{h_{t}}
=\frac{f''_{t}}{f'_{t}{}^{4}K^{2}_{t}}|\nabla f_{t}|^{2}_{h_{t}}.
\end{equation*}
Using these equations, we arrive at
\begin{eqnarray*}
\Box_{t}\left((f'_{t}K_{t})^{-1}\right)&=&\Box_{t}(f'_{t}{}^{-1})\cdot K^{-1}_{t}+f'_{t}{}^{-1}\cdot\Box_{t}(K^{-1}_{t})
+2\left\langle\nabla f'_{t}{}^{-1},\nabla K^{-1}_{t}\right\rangle_{h_{t}}\\
&=&\left(\frac{3f''_{t}{}^{2}}{K_{t}f'_{t}{}^{5}}-\frac{f'''_{t}}{K_{t}f'_{t}{}^{4}}\right)|\nabla f_{t}|^{2}_{h_{t}}
-\frac{f''_{t}}{K_{t}f'_{t}{}^{3}}\Box_{t} f_{t}\\
&&+\left(\frac{2}{f'_{t}{}^{3}K^{3}_{t}}+\frac{f''_{t}}{f'_{t}{}^{4}K^{2}_{t}}
\right)|\nabla f_{t}|^{2}_{h_{t}}-\frac{1}{f'_{t}{}^{2}K^{2}_{t}}\Box_{t} f_{t}
+\frac{2f''_{t}}{f'_{t}{}^{4}K^{2}_{t}}|\nabla f_{t}|^{2}_{h_{t}}\\
&=&-\left(1+\frac{f''_{t}K_{t}}{f'_{t}}\right)\frac{1}{f'_{t}{}^{2}K^{2}_{t}}
\Box_{t} f_{t}\\
&&+\left(\frac{3f''_{t}{}^{2}}{f'_{t}{}^{4}}-\frac{f'''_{t}}{f'_{t}{}^{3}}
+\frac{2}{f'_{t}{}^{2}K^{2}_{t}}+\frac{3f''_{t}}{f'_{t}{}^{3}K_{t}}\right)
\frac{|\nabla f_{t}|^{2}_{h_{t}}}{f'_{t}K_{t}}.
\end{eqnarray*}
The Laplacian of $(h_{t})^{-1}_{ij}$ with respect to $(h_{t})_{ij}$ is given by
\begin{eqnarray*}
\Box\left((h_{t})^{-1}_{ij}\right)&=&(h_{t})^{-1}_{k\ell}\nabla_{k}
\left(-(h_{t})^{-1}_{ip}(h_{t})^{-1}_{jq}\nabla_{\ell}(h_{t})_{pq}\right) \ \ = \ \ -(h_{t})^{-1}_{ip}(h_{t})^{-1}_{jq}\Box_{t} (h_{t})_{pq}\\
&&+(h_{t})^{-1}_{k\ell}(h_{t})^{-1}_{ip}(h_{t})^{-1}_{jr}h^{-1}_{qs}
\nabla_{k}(h_{t})_{rs}\nabla_{\ell}(h_{t})_{pq}\\
&&+(h_{t})^{-1}_{k\ell}(h_{t})^{-1}_{jq}(h_{t})^{-1}_{ir}(h_{t})^{-1}_{ps}
\nabla_{k}(h_{t})_{rs}\nabla_{\ell}(h_{t})_{pq}.
\end{eqnarray*}
So the second term in the expression of $\Box_{t}((f'_{t}K_{t})^{-1}
|\nabla f_{t}|^{2}_{h_{t}})$ is
\begin{eqnarray*}
\Box_{t}((h_{t})^{-1}_{ij})\nabla_{i}f_{t}\nabla_{j}f_{t}
&=&-(h_{t})^{-1}_{ip}(h_{t})^{-1}_{jq}\Box_{t} (h_{t})_{pq}\nabla_{i}f_{t}\nabla_{j}f_{t}\\
&&+2(h_{t})^{-1}_{k\ell}(h_{t})^{-1}_{ip}(h_{t})^{-1}_{jr}(h_{t})^{-1}_{qs}
\nabla_{k}(h_{t})_{rs}\nabla_{\ell}(h_{t})_{pq}\nabla_{i}f_{t}\nabla_{j}f_{t}\\
&=&-\left\langle\Box_{t} (h_{t})_{ij},\nabla_{i}f_{t}\nabla_{j}f_{t}
\right\rangle_{h_{t}}\\
&&+2\left\langle\nabla_{i}f_{t}\nabla_{j}(h_{t})_{k\ell},\nabla_{j}f_{t}
\cdot\nabla_{i}(h_{t})_{k\ell}\right\rangle_{h_{t}}.
\end{eqnarray*}
Combining those identities, we have
\begin{eqnarray*}
\Box_{t}\left(\frac{|\nabla f_{t}|^{2}_{h_{t}}}{f'_{t}K_{t}}\right)
&=&-\left(1+\frac{f''_{t}K_{t}}{f'_{t}}\right)\frac{\Box_{t} f_{t}|\nabla f_{t}|^{2}_{h_{t}}}{(f'_{t}K_{t})^{2}}-\frac{\left\langle\Box_{t} (h_{t})_{ij},\nabla_{i}f_{t}\nabla_{j}f_{t}
\right\rangle_{h_{t}}}{f'_{t}K_{t}}
\\
&&+\left(\frac{3f''_{t}{}^{2}}{f'_{t}{}^{4}}-\frac{f'''_{t}}{f'_{t}{}^{3}}
+\frac{2}{f'_{t}{}^{2}K^{2}_{t}}
+\frac{3f''_{t}}{f'_{t}{}^{3}K_{t}}\right)\frac{|\nabla f_{t}|^{4}_{h_{t}}}{f'_{t}K_{t}}\\
&&\\
&&+2\frac{\left\langle\nabla_{i}f_{t}\nabla_{i}(h_{t})_{k\ell},
\nabla_{j}f_{t}\cdot\nabla_{i}(h_{t})_{k\ell}\right\rangle_{h_{t}}}{f'_{t}K_{t}}+2\frac{|\nabla\nabla f_{t}|^{2}_{h_{t}}}{f'_{t}K_{t}}\\
&&+\frac{\left\langle\nabla_{i}(\Box_{t} f_{t})+\langle\nabla_{i}(h_{t})_{jk},\nabla_{j}\nabla_{k}f_{t}\rangle_{h_{t}},\nabla_{i}f_{t}
\right\rangle_{h_{t}}}{f'_{t}K_{t}}\\
&&+\frac{\left\langle
(n-1)(h_{t})_{ij}(g_{t})^{jk}\nabla_{k}f_{t},\nabla_{i}f_{t}\right
\rangle_{h_{t}}}{f'_{t}K_{t}}\\
&&+2\left\langle\nabla_{k}(f'_{t}K_{t})^{-1},
\nabla_{k}(h_{t})^{-1}_{ij}\right\rangle_{h_{t}}\nabla_{i}f_{t}
\cdot\nabla_{j}f_{t}\\
&&+4\left\langle\nabla_{k}(f'_{t}K_{t})^{-1}\nabla_{i}f_{t},
\nabla_{k}\nabla_{i}f_{t}\right\rangle_{h_{t}}\\
&&-4\frac{\left\langle\nabla_{k}(h_{t})_{ij},\nabla_{k}\nabla_{i}f_{t}
\cdot\nabla_{j}f_{t}\right\rangle_{h_{t}}}{f'_{t}K_{t}},
\end{eqnarray*}
and we also have
\begin{eqnarray*}
\Box_{t}\left(\frac{|\nabla f_{t}|^{2}_{h_{t}}}{f'_{t}K_{t}}\right)
&=&-\left(1+\frac{f''_{t}K_{t}}{f'_{t}}\right)\frac{1}{(f'_{t}K_{t})^{2}}\Box_{t}
f_{t}\cdot|\nabla f_{t}|^{2}_{h_{t}}\\
&&+\left(\frac{3f''_{t}{}^{2}}{f'_{t}{}^{4}}-\frac{f'''_{t}}{f'_{t}{}^{3}}
+\frac{2}{(f'_{t}K_{t})^{2}}+\frac{3f''_{t}}{f'_{t}{}^{3}K_{t}}
\right)\frac{|\nabla f_{t}|^{4}_{h_{t}}}{f'_{t}K_{t}}\\
&&-\frac{\left\langle\Box_{t} (h_{t})_{ij},
\nabla_{i}f_{t}\nabla_{j}f_{t}\right\rangle_{h_{t}}}{f'_{t}K_{t}}\\
&&+2\frac{\left\langle\nabla_{i}f_{t}\nabla_{j}(h_{t})_{kl},
\nabla_{j}f_{t}\nabla_{i}(h_{t})_{kl}\right\rangle_{h_{t}}}{f'_{t}K_{t}}\\
&&+\frac{2(n-1)|\nabla f_{t}|^{2}_{g_{t}}}{f'_{t}K_{t}}+2\frac{\left\langle\nabla(\Box_{t} f_{t}),\nabla f_{t}
\right\rangle_{h_{t}}}{f'_{t}K_{t}}\\
&&+2\left[\frac{f''_{t}}{f'_{t}{}^{3}K_{t}}+\frac{1}{(f'_{t}K_{t})^{2}}
\right]\left\langle\nabla_{i}(h_{t})_{jk},\nabla_{i}f_{t}\nabla_{j}f_{t}
\nabla_{k}f_{t}\right\rangle_{h_{t}}\\
&&-4\left[\frac{f''_{t}}{f'_{t}{}^{3}K_{t}}+\frac{1}{(f'_{t}K_{t})^{2}}
\right]\left\langle\nabla_{i}\nabla_{j}f_{t},\nabla_{i}f_{t}\nabla_{j}f_{t}
\right\rangle_{h_{t}}\\
&&-2\frac{\left\langle\nabla_{i}(h_{t})_{jk},\nabla_{i}f_{t}
\nabla_{j}\nabla_{k}f_{t}
\right\rangle_{h_{t}}}{f'_{t}K_{t}}
+2\frac{|\nabla\nabla f_{t}|^{2}_{h_{t}}}{f'_{t}K_{t}},
\end{eqnarray*}
where we use the identities
\begin{eqnarray*}
\nabla\left((f'_{t}K_{t})^{-1}\right)&=&-\left(\frac{f''_{t}}{f'_{t}{}^{3}K_{t}}
+\frac{1}{f'_{t}{}^{2}K^{2}_{t}}\right)\nabla f_{t},\\
\nabla_{k}(h_{t})^{-1}_{ij}&=&-(h_{t})^{-1}_{ip}(h_{t})^{-1}_{jq}\nabla_{k}(h_{t})_{pq}.
\end{eqnarray*}
From the above equations we obtain
\begin{eqnarray*}
I&=&\left(1
+\frac{f''_{t}K_{t}}{f'_{t}}\right)
\frac{(2\Box_{t} f_{t}+H_{t}f_{t})|\nabla f_{t}|^{2}_{h_{t}}}{f'_{t}K_{t}}\\
&+&\left[4\left(\frac{f''_{t}K_{t}}{f'_{t}}+1\right)-1\right]\frac{1}{f'_{t}K_{t}}
\left\langle\nabla_{i}\nabla_{j}f_{t},\nabla_{i}f_{t}\cdot\nabla_{j}f_{t}
\right\rangle_{h_{t}}\\
&+&\left(\frac{f_{t}}{f'_{t}K_{t}}-2(n-1)\right)|\nabla f_{t}|^{2}_{g_{t}}+2\langle\nabla(H_{t}f_{t}),\nabla f_{t}\rangle_{h_{t}}\\
&-&\left[\frac{3f''_{t}{}^{2}}{f'_{t}{}^{4}}-\frac{f'''_{t}}{f'_{t}{}^{3}}+\frac{2}{(f'_{t}K_{t})^{2}}
+\frac{3f''_{t}}{f'_{t}{}^{3}K_{t}}\right]|\nabla f_{t}|^{4}_{h_{t}}\\
&+&\left\langle\Box_{t} (h_{t})_{ij},\nabla_{i}f_{t}\cdot\nabla_{j}f_{t}\right\rangle_{h_{t}}
-2\left\langle\nabla_{i}f_{t}\cdot\nabla_{j}(h_{t})_{k\ell},\nabla_{j}f_{t}
\cdot\nabla_{i}(h_{t})_{k\ell}\right\rangle_{h}\\
&-&2\left(\frac{f''_{t}}{f'_{t}{}^{2}}+\frac{1}{f'_{t}K_{t}}
\right)\left\langle\nabla_{i}(h_{t})_{jk},\nabla_{i}f_{t}\cdot\nabla_{j}f_{t}
\cdot\nabla_{k}f_{t}\right\rangle_{h_{t}}\\
&+&2\left\langle\nabla_{i}(h_{t})_{jk},\nabla_{i}f_{t}\cdot\nabla_{j}
\nabla_{k}f_{t}
\right\rangle_{h_{t}}-2|\nabla\nabla f_{t}|^{2}_{h_{t}},
\end{eqnarray*}
where
\begin{equation*}
I:=\partial_{t}\left(\frac{|\nabla f_{t}|^{2}_{h_{t}}}{f'_{t}K_{t}}\right)
-(f'_{t}K_{t})\Box_{t}\left(\frac{|\nabla f_{t}|^{2}_{h_{t}}}{f'_{t}K_{t}}
\right).
\end{equation*}
On the other hand, from \cite{C1}, we get
\begin{eqnarray*}
\Box_{t}(h_{t})_{ij}&=&\frac{1}{f'_{t}K_{t}}\nabla_{i}\nabla_{j}f_{t}
-\left(1+\frac{f''_{t}K_{t}}{f'_{t}}\right)\frac{1}{(f'_{t}K_{t})^{2}}
\nabla_{i}f_{t}\nabla_{j}f_{t}\\
&+&\left\langle\nabla_{i}(h)_{k\ell},
\nabla_{j}(h)_{k\ell}\right\rangle_{h_{t}}-H_{t}(h_{t})_{ij}+n(g_{t})^{k\ell}
(h_{t})_{ik}(h_{t})_{\ell j}.
\end{eqnarray*}
Here we use the identity $\nabla_{i}\nabla_{j}f_{t}=f''_{t}\nabla_{i}K_{t}
\nabla_{j}K_{t}+f'_{t}\nabla_{i}\nabla_{j}K_{t}$. Plugging it into previous formula, we get
\begin{eqnarray*}
I&=&\left(1+\frac{f''_{t}K_{t}}{f'_{t}}\right)
\frac{2\Box_{t} f_{t}|\nabla f_{t}|^{2}_{h_{t}}}{f'_{t}K_{t}}
+4\left(1+\frac{f''_{t}K_{t}}{f'_{t}}\right)\frac{1}{f'_{t}K_{t}}
\left\langle\nabla_{i}\nabla_{j}f_{t},\nabla_{i}f_{t}\nabla_{j}f_{t}
\right\rangle_{h_{t}}\\
&+&\left[\frac{f_{t}}{f'_{t}K_{t}}-2(n-1)+n\right]|\nabla f_{t}|^{2}_{g_{t}}
+\left[1+\left(1+\frac{f''_{t}K_{t}}{f'_{t}}\right)\frac{f_{t}}{f'_{t}K_{t}}\right]H_{t}|\nabla f_{t}|^{2}_{h_{t}}\\
&+&2f_{t}\langle\nabla H_{t},\nabla f_{t}\rangle_{h_{t}}-\left(\frac{3f''_{t}{}^{2}}{f'_{t}{}^{4}}
-\frac{f'''_{t}}{f'_{t}{}^{3}}
+\frac{3}{(f'_{t}K_{t})^{2}}+\frac{4f''_{t}}{f'_{t}{}^{3}K_{t}}
\right)|\nabla f_{t}|^{4}_{h_{t}}\\
&-&\left\langle\nabla_{i}f_{t}\cdot\nabla_{j}(h_{t})_{k\ell},\nabla_{j}(f_{t})
\cdot\nabla_{i}
(h_{t})_{k\ell}\right\rangle_{h_{t}}\\
&-&2\left(\frac{f''_{t}K_{t}}{f'_{t}}+1\right)\frac{1}{f'_{t}K_{t}}
\left\langle\nabla_{i}(h_{t})_{jk},
\nabla_{i}f_{t}\nabla_{j}f_{t}\nabla_{k}f_{t}\right\rangle_{h_{t}}\\
&+&2\left\langle\nabla_{i}(h_{t})_{jk},\nabla_{i}f_{t}\nabla_{j}\nabla_{k}f_{t}
\right\rangle_{h_{t}}-2|\nabla\nabla f_{t}|^{2}_{h_{t}}.
\end{eqnarray*}
The final step is to compute $\langle\nabla_{i}\nabla_{j}f_{t},\nabla_{i}f_{t}\nabla_{j}f_{t}
\rangle_{h_{t}}$. We consider
\begin{eqnarray*}
\left\langle\nabla f_{t},\nabla\left(\frac{|\nabla f_{t}|^{2}
_{h_{t}}}{f'_{t}K_{t}}\right)
\right\rangle_{h_{t}}&=&-\left(\frac{f''_{t}}{f'_{t}{}^{3}K_{t}}+\frac{1}{K^{2}_{t}f'_{t}{}^{2}}
\right)|\nabla f_{t}|^{4}_{h_{t}}\\
&&-\frac{1}{f'_{t}K_{t}}\left\langle\nabla_{i}(h_{t})_{jk},
\nabla_{i}f_{t}\nabla_{j}
f_{t}\nabla_{k}f_{t}\right\rangle_{h_{t}}\\
&&+\frac{2}{f'_{t}K_{t}}\left\langle\nabla_{i}\nabla_{j}f_{t},
\nabla_{i}f_{t}\nabla_{j}f_{t}\right\rangle_{h_{t}}.
\end{eqnarray*}
Substituting this formula into the evolution equation, we obtain
\begin{eqnarray*}
I&=&2\left(1+\frac{f''_{t}K_{t}}{f'_{t}}\right)\left\langle\nabla f_{t},\nabla\left((f'_{t}K_{t})^{-1}|\nabla f_{t}|^{2}_{h_{t}}\right)\right\rangle_{h_{t}}-2|\nabla\nabla f_{t}|^{2}_{h_{t}}\\
&&+2\left\langle\nabla_{i}(h_{t})_{jk},\nabla_{i}f_{t}
\nabla_{j}\nabla_{k}f_{t}\right\rangle_{h_{t}}
+\left(1+\frac{f''_{t}K_{t}}{f'_{t}}\right)\frac{2}{f'_{t}K_{t}}\Box_{t} f_{t}
\cdot|\nabla f_{t}|^{2}_{h_{t}}\\
&&-\left\langle\nabla_{i}f_{t}\nabla_{j}(h_{t})_{kl},\nabla_{j}f_{t}
\nabla_{i}(h_{t})_{kl}\right\rangle_{h_{t}}
+\left(\frac{f_{t}}{f'_{t}K_{t}}+2-n\right)|\nabla f_{t}|^{2}_{g_{t}}\\
&+&\left[1+\left(1+\frac{f''_{t}K_{t}}{f'_{t}}\right)\frac{f_{t}}{f'_{t}K_{t}}
\right]H_{t}|\nabla f_{t}|^{2}_{h_{t}}+2f_{t}\langle\nabla H_{t},
\nabla f_{t}\rangle_{h_{t}}\\
&&-\left[\frac{3f''_{t}{}^{2}}{f'_{t}{}^{4}}-\frac{f'''_{t}}{f'_{t}{}^{3}}
+\frac{3}{(f'_{t}K_{t})^{2}}+\frac{4f''_{t}}{f'_{t}{}^{3}K}
\right.\\
&&-\left.2\left(1+\frac{f''_{t}K_{t}}{f'_{t}}\right)\left(\frac{f''_{t}}
{f'_{t}{}^{3}K_{t}}+\frac{1}{(f'_{t}K_{t})^{2}}\right)\right]|\nabla f_{t}|^{4}_{h_{t}}.
\end{eqnarray*}
The bracket in the last term equals $\frac{f''_{t}{}^{2}}{f'_{t}{}^{4}}
-\frac{f'''_{t}}{f'_{t}{}^{3}}
+\frac{1}{(f'_{t}K_{t})^{2}}$.
\end{proof}

\begin{lemma} \label{l3.3}Under the $f$-Gaussian curvature flow, we have
\begin{equation}
\partial_{t}(f_{t}H_{t})=f'_{t}K_{t}\cdot(\Box_{t} f_{t}+H_{t}f_{t})H_{t}
+f_{t}(\Delta_{t} f_{t}+f_{t}|h_{t}|^{2}_{g_{t}}).\label{3.8}
\end{equation}
\end{lemma}

\begin{proof} This immediately follows from (\ref{2.19}) and (\ref{2.20}).
\end{proof}

\begin{lemma} \label{l3.4} Under the $f$-Gaussian curvature flow, we have
\begin{eqnarray}
|P_{t}|^{2}_{h_{t}}&=&|\nabla\nabla f_{t}|^{2}_{h_{t}}+\left\langle\nabla_{i}f_{t}\cdot\nabla_{j}(h_{t})_{kl},
\nabla_{j}f_{t}\cdot\nabla_{i}(h_{t})_{k\ell}\right\rangle_{h_{t}}\nonumber
+f^{2}_{t}|h_{t}|^{2}\\
&&-2\left\langle\nabla_{i}(h_{t})_{jk},\nabla_{i}f_{t}
\cdot\nabla_{j}
\nabla_{k}f_{t}\right\rangle_{h_{t}}+2f_{t}\Delta_{t} f_{t}-2f_{t}
\langle\nabla H_{t},\nabla f_{t}\rangle_{h_{t}},\label{3.9}\\
\mathbf{P}^{2}_{t}&=&\left(\Box_{t} f_{t}\right)^{2}+\frac{|\nabla f_{t}|^{4}_{h_{t}}}{(f'_{t}K_{t})^{2}}
+f^{2}_{t}H^{2}_{t}-2\frac{\Box_{t} f_{t}\cdot|\nabla f_{t}|^{2}_{h_{t}}}{f'_{t}K_{t}}\nonumber\\
&&+2H_{t}f_{t}\cdot\Box_{t} f_{t}-2\frac{H_{t}f_{t}|\nabla f_{t}|^{2}_{h_{t}}}{f'_{t}K_{t}}.\label{3.10}
\end{eqnarray}
\end{lemma}

\begin{proof} By definition, we have
\begin{eqnarray*}
|P|^{2}_{h}&=&(h_{t})^{-1}_{ik}(h_{t})^{-1}_{j\ell}\left[\nabla_{i}\nabla_{j}f_{t}\cdot
\nabla_{k}\nabla_{\ell}f_{t}-(h_{t})^{-1}_{pq}\nabla_{p}(h_{t})_{ij}
\nabla_{q}f_{t}\nabla_{k}\nabla_{\ell}f_{t}\right.\\
&&-f_{t}(g_{t})^{pq}(h)_{ip}(h_{t})_{jq}\nabla_{k}\nabla_{\ell}f_{t}
-(h_{t})^{-1}_{pq}\nabla_{p}(h_{t})_{k\ell}\nabla_{q}f_{t}\nabla_{i}
\nabla_{j}f_{t}\\
&&+(h_{t})^{-1}_{pq}(h_{t})^{-1}_{rs}\nabla_{p}(h_{t})_{ij}
\nabla_{r}(h_{t})_{k\ell}\nabla_{q}f_{t}\nabla_{s}f_{t}\\
&&-f_{t}(g_{t})^{rs}(h_{t})_{ir}(h_{t})_{js}(h_{t})^{-1}_{pq}
\nabla_{p}(h_{t})_{k\ell}\nabla_{q}f_{t}+\left.f_{t}(g_{t})^{pq}(h_{t})_{kp}(h_{t})_{\ell q}\nabla_{i}\nabla_{j}f_{t}\right.\\
&&-\left.f_{t}(g_{t})^{rs}(h_{t})_{kr}(h_{t})_{\ell s}(h_{t})^{-1}_{pq}\nabla_{p}(h_{t})_{ij}\nabla_{q}f_{t}\right.\\
&&+\left.f^{2}(g_{t})^{pq}(h_{t})_{ip}(h_{t})_{jq}(g_{t})^{rs}(h_{t})_{kr}(h_{t})_{\ell s}
\right]\\
&=&|\nabla\nabla f_{t}|^{2}_{h_{t}}+\left\langle\nabla_{i}f_{t}\nabla_{j}
(h_{t})_{k\ell},
\nabla_{j}f_{t}\nabla_{i}(h_{t})_{k\ell}\right\rangle_{h_{t}}
+f^{2}_{t}|h_{t}|^{2}_{g_{t}}\\
&&-2\left\langle\nabla_{i}(h_{t})_{jk},\nabla_{i}f_{t}\nabla_{j}\nabla_{k}f_{t}
\right\rangle_{h_{t}}+2f_{t}\Delta_{t} f_{t}-2f_{t}\langle\nabla H_{t},
\nabla f_{t}\rangle_{h_{t}}.
\end{eqnarray*}
The second equation is obviously proved by using the first one.
\end{proof}

As a direct consequence, we obtain

\begin{theorem} \label{t3.5}Under the $f$-Gaussian curvature flow, we have
\begin{eqnarray}
\partial_{t}\mathbf{P}_{t}&=&f'_{t}K_{t}\cdot\Box_{t}\mathbf{P}_{t}
+2\left(1+\frac{f''_{t}K_{t}}{f'_{t}}\right)\langle\nabla f_{t},
\nabla \mathbf{P}_{t}\rangle_{h_{t}}
+|P_{t}|^{2}_{h_{t}}+\left(1+\frac{f''_{t}K_{t}}{f'_{t}}\right)\mathbf{P}^{2}_{t}\nonumber\\
&&+\left[\left(\frac{f''_{t}K_{t}}{f'_{t}}\right)^{2}-\frac{f''_{t}K_{t}}{f'_{t}}
-\frac{f'''_{t}K^{2}_{t}}{f'_{t}}\right]
\frac{|\nabla f_{t}|^{2}_{h_{t}}}{f'_{t}K_{t}}\left(\frac{|\nabla f_{t}|^{2}_{h_{t}}}{f'_{t}K_{t}}-\Box_{t} f_{t}
-H_{t}f_{t}\right)\nonumber\\
&&+\left(f'_{t}K_{t}-\frac{f_{t}f''_{t}K_{t}}{f'_{t}}-f_{t}\right)(H^{2}_{t}f_{t}
+H_{t}\cdot\Box_{t} f_{t})\nonumber\\
&&+\left[\left(1+\frac{f''_{t}K_{t}}{f'_{t}}\right)\frac{f_{t}}{f'_{t}K_{t}}
-1\right]H_{t}|\nabla f_{t}|^{2}_{h_{t}}.
\label{3.11}
\end{eqnarray}
\end{theorem}

%%%%%%%%%%%%%%%%%%%%%%%%%%%%%%%%%%%%%%%%%%%%%%%%%%%%%%%%%%%%%%%%%%%%%%%%%%%%%%
\subsection{Harnack inequality for the negative power Gaussian curvature flow}
%%%%%%%%%%%%%%%%%%%%%%%%%%%%%%%%%%%%%%%%%%%%%%%%%%%%%%%%%%%%%%%%%%%%%%%%%%%%%%
For the sake of studying, we define three functions for $x>0$:
\begin{eqnarray}
\alpha(x)&=&\left(\frac{xf''(x)}{f'(x)}\right)^{2}-\frac{xf''(x)}{f'(x)}
-\frac{x^{2}f'''(x)}{f'(x)},\label{3.12}\\
\beta(x)&=&xf'(x)-\frac{xf(x)f''(x)}{f'(x)}-f(x),\label{3.13}\\
\gamma(x)&=&\left(1+\frac{xf''(x)}{f'(x)}\right)\frac{f(x)}{xf'(x)}-1.\label{3.13}
\end{eqnarray}
Using this simple notation, we can rewrite the evolution equation for $\mathbf{P}_{t}$ as
\begin{eqnarray*}
\partial_{t}\mathbf{P}_{t}&=&f'_{t}K_{t}\cdot\Box_{t} \mathbf{P}_{t}+2\left(1+\frac{f''_{t}K_{t}}{f'_{t}}\right)\langle\nabla f_{f},
\nabla\mathbf{P}_{t}\rangle_{h_{t}}
+|P|^{2}_{h_{t}}+\left(1+\frac{f''_{t}K_{t}}{f'_{t}}\right)\mathbf{P}^{2}_{t}\\
&&+\alpha(K_{t})\frac{|\nabla f_{t}|^{2}_{h_{t}}}{f'_{t}K_{t}}
\left(\frac{|\nabla f_{t}|^{2}_{h_{t}}}{f'_{t}K_{t}}-\Box_{t} f_{t}-H_{t}f_{t}
\right)\\
&&+\beta(K_{t})\left[H^{2}_{t}f_{t}+H_{t}\cdot\Box_{t} f_{t}\right]
+\gamma(K_{t})H_{t}|\nabla f_{t}|^{2}_{h_{t}}\\
&=&f'_{t}K_{t}\cdot\Box_{t}\mathbf{P}_{t}+2\left(1+\frac{f''_{t}K_{t}}{f'_{t}}
\right)\langle\nabla f_{t},\nabla \mathbf{P}_{t}\rangle_{h_{t}}
+|P_{t}|^{2}_{h_{t}}+\left(1+\frac{f''_{t}K_{t}}{f'_{t}}\right)\mathbf{P}^{2}_{t}\\
&&+\left(H_{t}\beta(K_{t})-\alpha(K_{t})\frac{|\nabla f_{t}|^{2}_{h_{t}}}{f'_{t}K_{t}}\right)\mathbf{P}_{t}+
\left(\frac{\beta(K_{t})}{f'_{t}K_{t}}+\gamma(K_{t})\right)H_{t}|\nabla f_{t}|^{2}_{h_{t}}.
\end{eqnarray*}
Observing that $\gamma(x)=-\beta(x)/xf(x)'$ and $\beta'(x)=f(x)\alpha(x)/x$, we have
\begin{eqnarray}
\partial_{t}\mathbf{P}_{t}&=&f'_{t}K_{t}\cdot\Box P+2\left(1+\frac{f''_{t}K_{t}}{f'_{t}}\right)\langle\nabla f_{t},
\nabla \mathbf{P}_{t}\rangle_{h_{t}}
+|P_{t}|^{2}_{h_{t}}+\left(1+\frac{f''_{t}K_{t}}{f'_{t}}\right)\mathbf{P}^{2}_{t}\nonumber\\
&&+\left(H_{t}\beta_{t}-\frac{\beta'_{t}}{f_{t}f'_{t}}|\nabla f_{t}|^{2}_{h_{t}}\right)\mathbf{P}_{t},\label{3.15}
\end{eqnarray}
where $\beta_{t}=\beta(K_{t})$.

To obtain the Harnack inequality for the negative power Gaussian curvature flow,
we should impose some natural condition on $f$. First we investigate some properties of above three functions associated to the function $f$.

\begin{lemma} \label{l3.6}We have
\begin{itemize}

\item[(a)] $\alpha\equiv0$ if and only if $f'(x)=ax^{b}$ for some $a>0$ and $b\in\mathbb{R}$;

\item[(b)] $\beta\equiv0$ if and only if $f(x)=ax^{b}$ for $ab>0$.

\end{itemize}
\end{lemma}

\begin{proof} Suppose $\alpha\equiv0$. Then $x(f''^{2}-f'f''')=f''f'$ and hence
\begin{equation*}
x\left(\frac{f''^{2}-f'f''}{f'^{2}}\right)=f''f',
\end{equation*}
which implies that $-x(f''/f')'=f''/f'$. Let $g:=f''/f'$; so $xg'=-g$. Solving this ODE, we get $g=b/x$ for some constant $b$. For (b), putting $g=f'/f$ we get $f=ax^{b}$.
\end{proof}

When $f(x)=x^{b}$ for $b>0$, B. Chow \cite{C2} derived the Harnack inequality
for the $f$-Gaussian curvature flow. For the case $b<0$, we give the following:

\begin{theorem} \label{t3.7} If $f(x)=ax^{b}$ satisfies (1) $a>0$ and $b>0$, or (2) $a<0$ and $-\frac{1}{n}<b<0$, then
\begin{equation}
\mathbf{P}_{t}\geq-\frac{1}{\left(\frac{1}{n}+b\right)t}.\label{3.16}
\end{equation}
Consequently,
\begin{equation}
\frac{\partial f(K_{t})}{\partial t}-|\nabla f(K_{t})|^{2}_{h}+\frac{f'(K_{t})K_{t}}
{\left(\frac{1}{n}+b\right)t}\geq0.\label{3.17}
\end{equation}
\end{theorem}

\begin{proof} From the above lemma, we have
\begin{equation*}
\partial_{t}\mathbf{P}_{t}=abK^{b}_{t}\cdot\Box_{t} \mathbf{P}_{t}
+2b\langle\nabla f_{t},\nabla \mathbf{P}_{t}\rangle_{h_{t}}
+|P_{t}|^{2}_{h_{t}}+b\mathbf{P}^{2}_{t}.
\end{equation*}
Since $|P_{t}|^{2}_{h_{t}}\geq\frac{\mathbf{P}^{2}_{t}}{n}$, it follows that
\begin{equation*}
\partial_{t}\mathbf{P}_{t}\geq abK^{b}_{t}\cdot\Box_{t}\mathbf{P}_{t}
+2b\langle\nabla f,_{f}\nabla \mathbf{P}_{t}\rangle_{h_{t}}+\left(\frac{1}{n}+b\right)\mathbf{P}^{2}_{t}.
\end{equation*}
The parabolic maximum principle tells us that
\begin{equation*}
\mathbf{P}_{t}\geq-\frac{1}{\left(\frac{1}{n}+b\right)t}.
\end{equation*}
The last inequality is followed by $\partial_{t}f_{t}=f'_{t}K_{t}(\Box_{t} f_{t}
+f_{t}H_{t})$.
\end{proof}

Now, Theorem \ref{t1.1} follows from the above theorem.

%%%%%%%%%%%%%%%%%%%%%%%%%%%%%%%%%%%%%%%%%%%%%%%%%%%%%%%%%%%%%%%%%%%%%%%%%%%%%%
\bibliographystyle{amsplain}

\end{document}